\documentclass[10pt]{amsart}
\usepackage{amsmath,amssymb,amscd,graphicx,psfrag,epsf,amsthm
}
\usepackage[all]{xy}

\newtheorem{thm}{Theorem}[section]

\newtheorem{lemma}[thm]{Lemma}

\newtheorem{prop}[thm]{Proposition}

\newtheorem{conj}[thm]{Conjecture}
   
\theoremstyle{definition}
\newtheorem{defn}[thm]{Definition}

\newtheorem{say}[thm]{}

\newtheorem{defn-thm}[thm]{Definition-Theorem}  
  
\newtheorem{claim}[thm]{Claim} 
\newtheorem{rem}[thm]{Remark} 
\newtheorem{rems}[thm]{Remarks}

\theoremstyle{remark}

\renewcommand{\o}[0]{{\mathcal O}}  
\newcommand{\z}[0]{{\mathbb Z}}

\renewcommand{\r}[0]{{\mathbb R}}

\newcommand{\ch}{\text {ch}}
\newcommand{\td}{\text {td}}

\newcommand{\cC}{{\mathcal {C}}}

\newcommand{\cS}{{\mathcal {S}}}

\newcommand{\cN}{{\mathcal {N}}}

\newcommand{\p}[0]{{\mathbb P}}

\newcommand{\q}[0]{{\mathbb Q}}

\newcommand{\ev}[0]{\operatorname{ev}}
\newcommand{\id}[0]{\operatorname{id}}
\newcommand{\rat}[0]{\operatorname{RatCurves}^n}
\newcommand{\pic}[0]{\operatorname{Pic}}
\newcommand{\rank}[0]{\operatorname{rank}}

\newcommand{\Hom}[0]{\operatorname{Hom}}

\newcommand{\aut}[0]{\operatorname{Aut}}

\newcommand{\chow}[0]{\operatorname{Chow}}

\newcommand{\nec}[1]{\overline{NE}_1({#1})}
\newcommand{\eff}[0]{\overline{NE}}

\newcommand{\sym}[0]{\operatorname{Sym}}

\def\into{\DOTSB\lhook\joinrel\rightarrow}

\numberwithin{equation}{section}

\title[Rational curves on higher Fano manifolds]
{Polarized minimal families of rational curves \\ and higher Fano manifolds}

\date{}

\author{Carolina Araujo}
\author{Ana-Maria Castravet} 

\address{Carolina Araujo: \sf IMPA, Estrada Dona Castorina 110, Rio de
  Janeiro, RJ 22460-320, Brazil} 
\email{caraujo@impa.br}

\address{\vskip -.5cm Ana-Maria Castravet: \sf Department of Mathematics, 
University of Arizona, 617 N Santa Rita Ave, Tucson, AZ 85721-0089, USA} 
\email{noni@math.arizona.edu}

\begin{document}

\maketitle

\def\thefootnote{}
\footnotetext{2000 {\it{Mathematics Subject Classification}} Primary 14J45, Secondary 14M20.}

\begin{abstract}
In this paper we investigate  Fano manifolds $X$
whose Chern characters $\ch_k(X)$ satisfy some positivity conditions.
Our approach is via the study of \emph{polarized minimal families of rational curves} $(H_x,L_x)$
through a general point $x\in X$. 
First we translate positivity properties of the Chern characters of $X$
into properties of the pair $(H_x,L_x)$.
This allows us to classify polarized minimal families of rational curves
associated to Fano manifolds $X$ satisfying $\ch_2(X)\geq0$ and $\ch_3(X)\geq0$.
As a first application, we provide sufficient conditions for these manifolds
to be covered by subvarieties isomorphic to $\p^2$ and $\p^3$. 
Moreover, this classification enables us to find 
new examples of Fano manifolds satisfying  $\ch_2(X)\geq0$.
\end{abstract}


\section{Introduction}\label{introduction}

Fano manifolds, i.e., smooth complex projective varieties with ample anticanonical class, 
play an important role in the classification of complex projective varieties.
In \cite{mori79}, Mori showed that Fano manifolds are uniruled, i.e.,
they contain rational curves through every point. 
Then he studied minimal dominating families of rational curves on Fano manifolds, and used them to 
characterize projective spaces as the only smooth projective varieties having ample tangent bundle. Since then, 
minimal dominating families of rational curves have been extensively investigated and proved to be a useful tool in the
study of uniruled varieties.

Let $X$ be a smooth complex projective uniruled variety, and $x\in X$ a general point. 
A \emph{minimal family of rational curves through $x$} is a \emph{smooth} and \emph{proper} 
irreducible component $H_x$
of the scheme $\rat(X,x)$ parametrizing rational curves on $X$ passing through $x$. 
There always exists such a family.
For instance, one can take $H_x$ to be an irreducible component of $\rat(X,x)$ parametrizing 
rational curves through $x$ having minimal degree with respect to some fixed
ample line bundle on $X$.
While we view $X$ as an abstract variety, 
$H_x$ comes with a natural polarization $L_x$, which can be defined as follows
(see Section~\ref{section:rat_curves} for details).
By \cite{kebekus}, there is a \emph{finite morphism} $\tau_x:  \ H_x  \to \p(T_xX^{^{\vee}})$ 
that sends a point parametrizing a curve smooth at $x$ to its tangent  direction at $x$. 
We set $L_x=\tau_x^*\o(1)$. 
We call the pair $(H_x, L_x)$ a \emph{polarized minimal family of rational curves through $x$}.
The image $\cC_x$ of $\tau_x$ is called the \emph{variety of minimal rational tangents} at $x$.
The natural projective embedding $\cC_x\subset \p(T_xX^{^{\vee}})$ has been 
successfully explored to investigate the geometry of Fano manifolds.
See \cite{hwang} and \cite{hwang_ICM} for an overview of applications of the
variety of minimal rational tangents.

In this paper we view $(H_x, L_x)$ as a \emph{smooth polarized variety}.
We start by giving a formula for all the Chern characters of the variety $H_x$ in terms of the 
Chern characters of $X$ and $c_1(L_x)$.
This illustrates the general principle that  the pair $(H_x, L_x)$ encodes many 
properties of the ambient variety $X$. 
In what follows $\pi_x:U_x\to H_x$ and $\ev_x:U_x\to X$
denote the universal family morphisms introduced in section~\ref{section:rat_curves}, and the $B_j$'s 
denote the Bernoulli numbers, defined by the formula
$\frac{x}{e^x-1}=\sum_{j=0}^{\infty} \frac{B_j}{j!} x^j$.

\begin{prop} \label{chern_characters}
Let $X$ be a smooth complex projective uniruled variety.
Let $(H_x, L_x)$ be a polarized  minimal family of rational curves through a general point $x\in X$.
For any $k\geq1$, the $k$-th Chern character of $H_x$ is given by the formula:
\begin{equation}\label{ch_k for H_x}
ch_k(H_x)=\sum_{j=0}^{k}\frac{(-1)^jB_j}{j!}c_1(L_x)^j{\pi_x}_*\ev_x^*\big(\ch_{k+1-j}(X)\big)-\frac{1}{k!}c_1(L_x)^k.
\end{equation}
When $k$ is $1$ or $2$ this becomes:
\begin{equation}\label{c_1 of H_x}
c_1(H_x)={\pi_x}_*\ev_x^*\big(\ch_2(X)\big)+\frac{d}{2}c_1(L_x), \ \ and
\end{equation}
\begin{equation}\label{ch_2 of H_x}
\ch_2(H_x)={\pi_x}_*\ev_x^*\big(\ch_3(X)\big)+\frac{1}{2}{\pi_x}_*\ev_x^*\big(\ch_2(X)\big)\cdot c_1(L_x)+\frac{d-4}{12}c_1(L_x)^2.
\end{equation}
\end{prop}

Formulas for  the first Chern class $c_1(H_x)$ were previously obtained in  \cite[Proposition 4.2]{druel_chern_classes} and 
\cite[Theorem 1.1]{dJ-S:Chern_classes}. However, $c_1(L_x)$ appears disguised in those formulas.

\medskip

Next we turn our attention to Fano manifolds $X$ whose Chern characters satisfy some positivity 
conditions. 
In order to state our main theorem, we introduce some notation. 
See section~\ref{section:rat_curves} for details. 
Given a positive integer $k$, we denote by $N_k(X)_{\r}$ the 
real vector space of $k$-cycles on $X$ modulo numerical equivalence. 
We denote by $\eff_k(X)\subset N_k(X)_{\r}$ the closure of the cone generated by effective $k$-cycles.
There is a linear map ${\ev_x}_*\pi_x^*:  N_k(H_x)_{\r}  \to N_{k+1}(X)_{\r}$.
A codimension $k$ cycle $\alpha\in A^k(X)\otimes_\z\q$ is 
\emph{weakly positive} (respectively \emph{nef}) if $\alpha\cdot \beta>0$ (respectively $\alpha\cdot \beta \geq 0$)
for every effective integral $k$-cycle $\beta\neq 0$.
In this case we write $\alpha>0$ (respectively $\alpha\geq0$).

One easily checks  that the only del Pezzo surface satisfying $\ch_2>0$ is $\p^2$.
In \cite{2Fano_3folds}, we go through the classification of Fano threefolds, and
check that the only ones satisfying $\ch_2>0$ are $\p^3$ and the smooth quadric hypersurface in $\p^4$.
In higher dimensions, Proposition~\ref{chern_characters} above
allows us to translate positivity properties of the Chern characters of $X$ into those of $H_x$, and 
classify polarized varieties $(H_x, L_x)$ associated to 
Fano manifolds $X$ with $\ch_2(X)\geq0$ and $\ch_3(X)\geq0$.
The following is our main theorem.

\begin{thm}\label{thm1} \label{thm2}
Let $X$ be a  Fano manifold. 
Let $(H_x, L_x)$ be a polarized minimal family of rational curves through a general point $x\in X$.
Set $d=\dim H_x$.
\begin{enumerate}
\item If $\ch_2(X)>0$ (respectively $\ch_2(X)\geq0$), then $-2K_{H_x}-dL_x$ is ample (respectively nef). 
	This necessary condition is also sufficient provided that ${\ev_x}_*\pi_x^*\big(\eff_1(H_x)\big)=\eff_2(X)$. 
\item  If $\ch_2(X)>0$, then $H_x$ is a Fano manifold with $\rho(H_x)=1$ 
	except when $(H_x, L_x)$ is isomorphic to one of the following:
\begin{enumerate} 
	\item $\Big(\p^m\times \p^m, p_{_1}^*\o(1)\otimes p_{_2}^*\o(1)\Big)$, with $d=2m$,
	\item $\Big(\p^{m+1}\times\p^{m} \ ,\  p_{_1}^*\o(1)\otimes p_{_2}^*\o(1)\Big)$, 
		with $d=2m+1$,
	\item $\Big(\p_{\p^{m+1}}\Big(\o(2)\oplus \o(1)^{^{\oplus m}}\Big) \ , \ \o_{\p}(1)\Big)$, 
		with $d=2m+1$,
	\item  $\Big(\p^{m}\times Q^{m+1} \ ,\  p_{_1}^*\o(1)\otimes p_{_2}^*\o(1)\Big)$, 
		with $d=2m+1$, or
	\item  $\Big(\p_{\p^{m+1}}\big(T_{\p^{m+1}}\big) \ , \ \o_{\p}(1)\Big)$, with $d=2m+1$.
\end{enumerate}
Morever, each of these exceptional pairs occurs as $(H_x, L_x)$ for some Fano manifold $X$ with  
$\ch_2(X)>0$.
\item If $\ch_2(X)>0$, $\ch_3(X)\geq0$ and $d\geq 2$, then $H_x$ is a Fano manifold
with $\rho(H_x)=1$ and $\ch_2(H_x)>0$. 
\end{enumerate}
\end{thm}

\begin{rems}  
\noindent {\bf (i)} Fano manifolds $X$ with $\ch_2(X)\geq 0$ were introduced in \cite{dJ-S:2fanos_1} and \cite{dJ-S:2fanos_2}.
In \cite{dJ-S:2fanos_1} de Jong and Starr described a few examples and many non-examples  of such manifolds. 
Roughly, the only examples of Fano manifolds with $\ch_2(X)>0$ in their list are
complete intersections of type $(d_1,\cdots, d_m)$ in $\p^n$, 
with $\sum d_i^2\leq n$, and the Grassmannians $G(k,2k)$ and $G(k,2k+1)$. 
Theorem~\ref{thm1} explains why, as pointed out in \cite{dJ-S:2fanos_1}, other examples  are difficult to find. 

\smallskip

\noindent {\bf (ii)} 
Eventually, one would hope to classify all Fano manifolds with weakly positive (or even nef) higher Chern characters.
Theorem~\ref{thm1} is a step in this direction.
In fact, many homogeneous spaces $X$ are characterized by their variety of minimal rational tangents
$\cC_x=\tau_x(H_x)\subset \p(T_xX^{^{\vee}})$ among Fano manifolds with Picard number one.
This is the case when $X$ is a Hermitian symmetric space or a homogeneous contact manifold
 {\cite[Main Theorem]{mok_05}}, or $X$ is the quotient of a complex simple Lie group by
 a maximal parabolic subgroup associated to a long simple root  {\cite[Main Theorem]{hong_hwang}}.
Notice, however, that $(H_x,L_x)$ carries less information than the embedding $\cC_x\subset \p(T_xX^{^{\vee}})$.
For instance, in Example~\ref{G2/P}, 
$X$ is the $5$-dimensional homogeneous space $G_2/P$,
$(H_x, L_x)\cong \big(\p^1,\o(3)\big)$, and $\cC_x$ is a twisted cubic in  $\p(T_xX^{^{\vee}})\cong \p^4$,
and thus degenerate.

\smallskip

\noindent {\bf (iii)}
In a forthcoming paper we classify polarized varieties $(H_x, L_x)$ associated to 
Fano manifolds $X$ satisfying $\ch_2(X)\geq0$. In this case, the list of pairs 
$(H_x, L_x)$ with  $\rho(H_x)>1$ is 
much longer than the one in Theorem~\ref{thm1}(2).

\end{rems}

By \cite{mori79}, Fano manifolds are covered by rational curves.
In \cite{dJ-S:2fanos_2}, de Jong and Starr considered the question whether there 
is a rational surface through a general point of a Fano manifold $X$ 
satisfying $\ch_2(X)\geq 0$. 
They showed that the answer is positive if the pseudoindex $i_X$ of $X$ is at least $3$.
The condition $i_X\geq 3$ implies that $\dim H_x\geq1$, and so
Theorem~\ref{thm1}(1) recovers their result. 
In fact, we can say a bit more:

\begin{thm}\label{thm3}
Let $X$ be a  Fano manifold, and
$(H_x, L_x)$ a polarized minimal family of rational curves through a general point $x\in X$.
Suppose that $\ch_2(X)\geq0$ and $d=\dim H_x\geq 1$.
\begin{enumerate}
\item (\cite{dJ-S:2fanos_2}) There is a rational surface through $x$.
\item If $\ch_2(X)> 0$ and $(H_x, L_x)\not\cong$
$\big(\p^d,\o(2)\big)$,  $\big(\p^1,\o(3)\big)$, then
there is a generically injective morphism $g:(\p^{2},p)\to (X,x)$
mapping lines through $p$ to curves parametrized by $H_x$.
\end{enumerate}
Suppose moreover that $\ch_2(X)> 0$, $\ch_3(X)\geq0$ and $d\geq 2$.
\begin{enumerate}
\setcounter{enumi}{2}
\item There is a rational $3$-fold through $x$, except possibly if $(H_x, L_x)\cong \big(\p^2,\o(2)\big)$
and $\cC_x=\tau_x(H_x)$ is singular.
\item  Let $(W_h,M_h)$ be a polarized minimal family of rational curves through a general point $h\in H_x$.
Suppose that $(H_x, L_x)\not\cong \big(\p^d,\o(2)\big)$ and
$(W_h,M_h)\not\cong$ $\big(\p^k,\o(2)\big)$,  $\big(\p^1,\o(3)\big)$.
Then  there is a generically injective morphism $h:(\p^{3},q)\to (X,x)$ mapping  
 lines through $q$ to curves parametrized by $H_x$.
\end{enumerate}
\end{thm}

\begin{rems}
\noindent {\bf (i)} We believe that the exception in Theorem~\ref{thm3}(3) does not occur.

\smallskip

\noindent {\bf (ii)}
If $H_x$ parametrizes lines under some projective embedding  $X\into \p^N$,
then the morphisms $g$ and $h$ from Theorem~\ref{thm3}(2) and (4)
map lines of $\p^{2}$ and $\p^3$ to lines of $\p^N$. 
Hence, they are isomorphisms onto their images.
\end{rems}

This paper is organized as follows. In  section~\ref{section:rat_curves} we 
introduce polarized minimal families of rational curves and study some of their basic
properties.
In section~\ref{section:Chern} we make a 
Chern class computation to prove  Proposition~\ref{chern_characters}. 
This is a key ingredient to the proof of Theorem~\ref{thm1}.
Theorems~\ref{thm1}
 and \ref{thm3} are proved in section~\ref{section:proofs}.
In section~\ref{examples}, 
we give new examples of Fano manifolds satisfying $\ch_2(X)\geq 0$.
In particular, 
we exhibit  Fano manifolds $X$ with  $\ch_2(X)>0$ realizing
each of the exceptional pairs 
in Therorems~\ref{thm1}.

\bigskip

\noindent {\bf Notation.}
  Throughout this paper we work over the field of complex
  numbers.
  We often identify vector bundles with their corresponding locally free subsheaves. 
  We also identify a divisor on a smooth projective variety $X$ with 
  its corresponding line bundle and its class in $\pic(X)$.
  Let $E$ be a vector bundle on a variety $X$.
  We denote the Grothendieck projectivization
  $\operatorname{Proj}_X(Sym (E))$ by $\p(E)$, and the tautological line bundle on  $\p(E)$
  by $\o_{\p(E)}(1)$, or simply $\o_{\p}(1)$.
  By a \emph{rational curve} we mean a \emph{proper rational curve}, unless otherwise noted.

\bigskip

\noindent {\it Acknowledgments.} 
Most of this work was developed while we were research members at the Mathematical Sciences Research Institute (MSRI) during
the 2009 program in Algebraic Geometry. We are grateful to MSRI and the organizers of the program 
for providing a very stimulating environment for our research and for the financial support. 
This work has benefitted from ideas and suggestions by 
J\'anos Koll\'ar and Jaros\l aw Wi\'sniewski. We thank them for their comments and interest in our work. 
We thank Izzet Coskun, Johan de Jong, Jason Starr and Jenia Tevelev 
for fruitful discussions on the subject of this paper.
The first named author was partially supported by CNPq-Brazil Research Fellowship and
L'Or\'eal-Brazil For Women in Science Fellowship.


\section{Polarized minimal families of rational curves}\label{section:rat_curves}

We refer to \cite[Chapters I and II]{kollar} for the basic theory of rational curves on complex projective varieties. 
See also \cite{debarre}.

Let $X$ be a smooth complex projective uniruled variety of dimension $n$. 
Let $x\in X$ be a \emph{general} point. 
There is a scheme $\rat(X,x)$ parametrizing rational curves on $X$ passing through $x$.
This scheme is constructed as the normalization of a certain 
subscheme of the Chow variety $\chow(X)$ parametrizing effective $1$-cycles on $X$.
We refer to \cite[II.2.11]{kollar} for details on the construction of $\rat(X,x)$.
An irreducible component $H_x$ of $\rat(X,x)$ is called a \emph{family of rational curves through $x$}.
It can also be described as follows. 
There is an irreducible open subscheme $V_x$ of the Hom scheme $\Hom(\p^1,X,o\mapsto x)$ parametrizing
morphisms $f:\p^1\to X$ such that $f(o)=x$ and $f_*[\p^1]$ is parametrized by $H_x$. 
Then $H_x$ is the quotient of $V_x$ by the natural action of the automorphism 
group $\aut(\p^1,o)$. 
Given a morphism $f:\p^1\to X$ parametrized by $V_x$, 
we use the same symbol $[f]$ to denote the element of $V_x$ corresponding to $f$,
and its image in $H_x$.
Since $x\in X$ is a general point,  both $V_x$ and $H_x$ are smooth, and
every morphism $f:\p^1\to X$  parametrized by $V_x$ is \emph{free}, i.e.,
$f^*T_X\ \cong \ \bigoplus_{i=1}^{\dim X}\o_{\p^1}(a_i)$,  with all $a_i\geq 0$.
From the universal properties of $\Hom(\p^1,X,o\mapsto x)$ and $\chow(X)$, 
we get a commutative diagram:

\begin{equation} \label{diagram_Hx}
\xymatrix{
\p^1 \times V_x \ar[d] \ar[r] \ar@/^0.8cm/[rrr]^{(t,[f])\mapsto f(t)} & U_x \ar[d]_{\pi_x} \ar[rr]^{\ev_x} & & X, \\
V_x \ar[r]  \ar@/^0.4cm/[u]^{\{o\}\times \id} & H_x \ar@/_0.4cm/[u]_{\sigma_x}
}
\end{equation}
where $\pi_x$ is a $\p^1$-bundle and $\sigma_x$ is 
the unique section of $\pi_x$ such that 
$\ev_x\big(\sigma_x(H_x)\big)=x$. 
We denote by the same symbol both $\sigma_x$ and its image  in $U_x$,
which equals the image of $\{o\}\times V_x$ in $U_x$.

In \cite[Proposition 3.7]{druel_chern_classes}, Druel gave the following description of the tangent 
bundle of $H_x$:
\begin{equation} \label{druel}
T_{H_x}\ \cong \ (\pi_x)_* \Big( \big( (\ev_x^*T_X)/T_{\pi_x} \big)(-\sigma_x) \Big),
\end{equation}
where the relative tangent sheaf $T_{\pi_x}=T_{U_x/H_x}$ is identified with its image 
under the map $d\ev_x: T_{U_x}\to \ev_x^*T_X$.

When $H_x$ is proper,
we call it a \emph{minimal family of rational curves through $x$}.

\begin{say}[Minimal families of rational curves]\label{Hx}
Let $H_x$ be a minimal family of rational curves through $x$.
For a general point $[f]\in H_x$, we have
$f^*T_X \cong \o(2)\oplus \o(1)^{\oplus d}\oplus \o^{\oplus n-d-1}$,
where $d=\dim H_x=\deg(f^*T_X)-2\leq n-1$ (see \cite[IV.2.9]{kollar}).
Moreover, $d=n-1$ if and only if $X\cong \p^n$ by \cite{CMSB}
(see also \cite{kebekus_on_CMSB}).
Let $H_x^{\text{Sing},x}$ denote the subvariety of $H_x$ parametrizing curves that are singular at $x$.
By a result of Miyaoka (\cite[V.3.7.5]{kollar}), if $Z\subset H_x\setminus H_x^{\text{Sing},x}$ is a \emph{proper}
subvariety, then $\ev_x\big|_{\pi_x^{-1}(Z)}: \pi_x^{-1}(Z)\to X$ is generically injective. 
In particular, if $H_x^{\text{Sing},x}=\emptyset$, then $\ev_x$ is birational onto its image.
By \cite[Theorem 3.3]{kebekus}, $H_x^{\text{Sing},x}$ is at most finite,
and every curve parametrized by $H_x$ is immersed at $x$.
\end{say}

\begin{say}[Polarized minimal families of rational curves]\ \label{describing_Lx}
Now we describe a natural polarization $L_x$ associated to a
minimal family $H_x$ of rational curves through $x$.
There is an inclusion of sheaves
\begin{equation}\label{tau}
\sigma_x^*\big(T_{\pi_x}\big)\ \into \ \sigma_x^*\ev_x^*T_X\ \cong \ T_xX\otimes \o_{H_x}.
\end{equation}
By \cite[Theorems 3.3 and 3.4]{kebekus}, the cokernel of this map is locally free, 
and defines a finite morphism  $\tau_x:  \ H_x  \to \p(T_xX^{^{\vee}})$.
By \cite{hwang_mok_birationality}, $\tau_x$ is birational onto its image.
Notice that $\tau_x$ sends a curve that is smooth at $x$ to its tangent  direction at $x$. 
Set $L_x=\tau_x^*\o(1)$. It is an ample and globally generated line bundle on $H_x$.
We call the pair $(H_x, L_x)$ a \emph{polarized minimal family of rational curves through $x$}.

The following description of $L_x$ from \cite[4.2]{druel_chern_classes} is very useful for computations.
Set $E_x=(\pi_x)_*\o_{U_x}(\sigma_x)$. Then $U_x\cong \p(E_x)$ over $H_x$,
and under this isomorphism $\o_{U_x}(\sigma_x)$ is identified with the tautological line bundle
$\o_{\p(E_x)}(1)$.
Notice also that $\sigma_x^*\big(\o_{U_x}(\sigma_x)\big) \cong \sigma_x^* \big(\cN_{\sigma_x/U_x}\big) 
\cong \sigma_x^*\big(T_{\pi_x}\big)$.
Therefore \eqref{tau} induces in isomorphism 
\begin{equation}\label{Lx=-normal}
L_x\cong \sigma_x^*\big({T_{\pi_x}}\big)^{-1}\cong \sigma_x^*\o_{U_x}(-\sigma_x) \cong \sigma_x^*\o_{\p(E_x)}(-1).
\end{equation}
By pulling back by $\sigma_x$, the Euler sequence
\begin{equation}\label{euler}
0 \ \to \ \ \o_{\p(E_x)}(1)\otimes\big({T_{\pi_x}}\big)^{-1} \ \to \ \pi_x^*E_x  \ \to \ \o_{\p(E_x)}(1)  \ \to \ 0
\end{equation}
induces an exact sequence
\begin{equation}\label{L}
0 \ \to \ \o_{H_x} \ \to \ E_x  \ \to \ L_x^{-1}  \ \to \ 0.
\end{equation}
This description of $L_x$ and the projection formula 
yield the following identities of cycles on $U_x$:
\begin{itemize}
\item[(i) ] $\ev_x^*(c_1(X))=(d+2)(\sigma_x+\pi_x^*c_1(L_x))$,
\item[(ii) ] $\sigma_x\cdot\ev_x^*(\gamma)=0$ for any $\gamma\in A^k(X)$, $k\geq1$, and
\item[(iii) ] $\sigma_x^2=-\sigma_x\cdot\pi_x^*c_1(L_x)$,
\end{itemize}
where, as before, $d=\dim H_x=\deg(f^*T_X)-2$ for any $[f]\in H_x$.
\end{say}

\begin{lemma} \label{f:P^k->X}
Let $X$ be a smooth complex projective uniruled variety. 
Let $(H_x, L_x)$ be a polarized minimal family of rational curves through a general point $x\in X$.
Suppose there is a subvariety $Z\subset H_x$ such that 
$(Z, L_x|_Z)\cong (\p^k,\o_{\p^k}(1))$.
\begin{enumerate}
	\item Then there is a finite morphism $g:(\p^{k+1},p)\to (X,x)$ 
		that maps lines through $p$ birationally
		to curves parametrized by $H_x$. 
	\item If moreover $Z\subset H_x\setminus H_x^{\text{Sing},x}$, then 
		$g$ is generically injective.
\end{enumerate}
\end{lemma}

\begin{proof}
Let $Z\subset H_x$ be a subvariety such that $(Z, L_x|_Z)\cong (\p^k,\o_{\p^k}(1))$.
Set $U_Z:=\pi_x^{-1}(Z)$, $\sigma_Z:=\sigma_x\cap U_Z$, and $E_Z:=E_x|_{Z}$.
By \ref{describing_Lx}, $U_Z$ is isomorphic to $\p(E_Z)$ over $Z$, 
and under this isomorphism $\o_{U_Z}(\sigma_Z)$ is identified with the tautological line bundle
$\o_{\p(E_Z)}(1)$.
By \eqref{L}, $E_Z\cong \o_{\p^k}\oplus \o_{\p^k}(-1)$. Thus $U_Z$ is isomorphic to the blowup 
of $\p^{k+1}$ at a point $p$, and under this isomorphism $\sigma_Z$  is identified with the exceptional divisor.
Since $\ev_x|_{U_Z}: U_Z\to X$ maps $\sigma_Z$ to $x$ and contracts nothing else,
it factors through a finite morphism  $g:\p^{k+1}\to X$ mapping $p$ to $x$.
The lines through $p$ on $\p^{k+1}$ are images of  fibers of $\pi_x$ over $Z$, and thus are mapped 
birationally to curves parametrized by $Z\subset H_x$.

If $Z\subset H_x\setminus H_x^{\text{Sing},x}$, then 
$g$ is generically injective by \cite[V.3.7.5]{kollar}.
\end{proof}

\begin{say}\label{H}
There is a scheme $\rat(X)$ parametrizing rational curves on $X$.
A \emph{minimal dominating family of rational curves on $X$} is an
irreducible component $H$ of $\rat(X)$ parametrizing a
family of rational curves that sweeps out a dense open subset of $X$, and satisfying the following condition.
For a general point  $x\in X$, the (possibly reducible) subvariety $H(x)$ of $H$ 
parametrizing curves through $x$ is proper.
In this case, for each irreducible component $H(x)^i$ of $H(x)$, there is a 
minimal family $H_x^i$ of rational curves through $x$ parametrizing the same curves as $H(x)^i$.
Moreover, $H_x^i$ is naturally isomorphic to the normalization of $H(x)^i$.
This follows from the construction of $\rat(X)$ and $\rat(X,x)$ in \cite[II.2.11]{kollar}.
If in addition $H$ is proper, then we say that it is an \emph{unsplit covering family of rational curves}.
This is the case, for instance, when the curves parametrized by $H$ have degree $1$ with respect to 
some ample line bundle on $X$.
\end{say}

We end this section by investigating the relationship between the Chow ring of a
Fano manifold and  that of its minimal families of rational curves. 

\begin{defn}\label{defn_cycles}
Let $X$ be a projective variety, and $k$ a non negative integer. 
We denote by $A_k(X)$ the group of $k$-cycles on $X$ modulo rational equivalence, and
by $A^k(X)$ the $k^{\text{th}}$ graded piece of the 
Chow ring $A^*(X)$ of $X$. 
Let $N_k(X)$ (respectively $N^k(X)$) be the quotient of $A_k(X)$
(respectively $A^k(X)$) by numerical equivalence. 
Then $N_k(X)$ and $N^k(X)$ are finitely generated Abelian groups, and intersection 
product induces a perfect pairing $N^k(X)\times N_k(X)\to \z$. 
For every $\z$-module $B$, set $N_k(X)_B:=N_k(X)\otimes B$ and $N^k(X)_B:=N^k(X)\otimes B$.
We denote by $\eff_k(X)\subset N_k(X)_{\r}$ the closure of the cone generated by effective $k$-cycles.

Let $\alpha \in N^k(X)_{\r}$. We say that $\alpha$ is 
\begin{enumerate}
	\item[$\bullet$] \emph{ample} if $\alpha=A^k$ for some ample $\r$-divisor $A$ on $X$;
	\item[$\bullet$] \emph{positive} if $\alpha\cdot \beta>0$ for every $\beta\in \eff_k(X)\setminus \{0\}$;
	\item[$\bullet$] \emph{weakly positive} if $\alpha\cdot \beta>0$ for every effective integral 
		$k$-cycle $\beta\neq 0$; 
	\item[$\bullet$] \emph{nef} if $\alpha\cdot \beta\geq0$ for every $\beta\in \eff_k(X)$.
\end{enumerate}
We write $\alpha>0$ for $\alpha$ weakly positive and $\alpha\geq 0$ for $\alpha$ nef.
\end{defn}

\begin{defn}\label{defn_Tk}
Let $X$ be a smooth projective uniruled variety, and 
$H_x$ a minimal family of rational curves through a general point $x\in X$.
Let $\pi_x$ and $\ev_x$ be as in \eqref{diagram_Hx}.
For every positive integer $k$,
we define linear maps 
\begin{align}
T^k: \ &  N^k(X)_{\r}  \ \to \ N^{k-1}(H_x)_{\r}\ , & \ T_k: \ & N_k(H_x)_{\r} \  \to \ N_{k+1}(X)_{\r}. \notag \\
           & \ \ \ \ \ \alpha \ \mapsto \ {\pi_x}_*\ev_x^*\alpha &   & \ \ \ \ \ \beta \ \mapsto \ {\ev_x}_*\pi_x^*\beta \notag
\end{align}
This is possible because $\ev_x$ is proper and $\pi_x$ is a $\p^1$-bundle, and thus flat. 
We remark that in general these maps are neither injective nor surjective.
\end{defn}

\begin{lemma}\label{Tk_preserves_positivity}
Let $X$ be a smooth projective uniruled variety, and 
$(H_x, L_x)$ a polarized minimal family of rational curves through a general point $x\in X$.
\begin{enumerate}
	\item Let $A$ be an $\r$-divisor on $X$, and set $a=\deg f^*A$, where $[f]\in H_x$.
		Then $T^k(A^k)= a^k c_1(L_x)^{k-1}$.		
	\item $T_k$ maps $\eff_k(H_x)\setminus \{0\}$ into $\eff_{k+1}(X)\setminus \{0\}$.
	\item $T^k$ preserves the properties of being ample, positive, 
		weakly positive and nef.
\end{enumerate}
\end{lemma}

\begin{proof}
To prove (1), let $A$ be an $\r$-divisor on $X$, and set $a=\deg f^*A$, where $[f]\in H_x$.
Using \eqref{Lx=-normal} and \ref{describing_Lx}(ii),
it is easy to see that $\ev_x^*A= a(\sigma_x+\pi_x^*L_x)$ in $N^1(U_x)$.
By \ref{describing_Lx}(iii) and the projection formula,
{\small \begin{align}
T^k(A^k)={\pi_x}_*\ev_x^*(A^k)&=
a^k{\pi_x}_*\left[\sum_{i=0}^k{{k}\choose{i}}\sigma_x^{k-i}\cdot \pi_x^*c_1(L_x)^i\right] \notag \\
&=a^k{\pi_x}_*\left[\left(\sum_{i=0}^{k-1}{{k}\choose{i}}(-1)^{k-i-1}\right)\sigma_x \cdot
\pi_x^*c_1(L_x)^{k-1} + \pi_x^*c_1(L_x)^{k}\right] \notag \\
&=a^k{\pi_x}_*\Big[\sigma_x\cdot \pi_x^*c_1(L_x)^{k-1} +  \pi_x^*c_1(L_x)^{k}\Big] 
= a^k c_1(L_x)^{k-1}.\notag
\end{align}}

Notice that $T_k$ maps effective cycles to effective cycles, inducing a linear map 
$T_k: \eff_k(H_x) \to\eff_{k+1}(X)$. By taking $A$ an ample divisor in (1) above, we see that
$T_k$ maps $\eff_k(H_x)\setminus \{0\}$ into $\eff_{k+1}(X)\setminus \{0\}$.

By the projection formula, $T^{k+1}(\alpha)\cdot\beta=\alpha\cdot T_k(\beta)$ for every
$\alpha\in  N^{k+1}(X)_{\r}$ and $\beta\in N_k(H_x)_{\r}$. Together with (1) and (2) above, this
implies that $T^k$ preserves the properties of being ample, positive, 
weakly positive and nef.
\end{proof}


\section{A Chern class computation}\label{section:Chern}

In this section we prove Proposition~\ref{chern_characters}.
We refer to \cite{fulton} for basic results about intersection theory.
In particular, 
$\ch(F)$ denotes the Chern character of the sheaf $F$, and $\td(F)$ denotes its Todd class.
We follow the lines of the proof of \cite[Proposition 4.2]{druel_chern_classes}.

\begin{proof}[Proof of Proposition~\ref{chern_characters}]
We use the notation introduced in Section \ref{section:rat_curves}. 

By Grothendieck-Riemann-Roch
\begin{gather*}
\ch\Big({\pi_x}_!\big(\ev_x^*T_X/T_{\pi_x}(-\sigma_x)\big)\Big)= \\
{\pi_x}_*\Big(\ch\big(\ev_x^*T_X/T_{\pi_x}(-\sigma_x)\big)\cdot\td(T_{\pi_x})\Big)\in A(H_x)_{\q}.
\end{gather*}

Since $f:\p^1\to X$ is free for any $[f]\in H_x$, $R^1{\pi_x}_*\big(\ev_x^*T_X/T_{\pi_x}(-\sigma_x)\big)=0$, and thus
$
{\pi_x}_!\big(\ev_x^*T_X/T_{\pi_x}(-\sigma_x)\big)={\pi_x}_*\big(\ev_x^*T_X/T_{\pi_x}(-\sigma_x)\big).
$
It follows from \eqref{druel} that
$$
ch_k(H_x)=\ch_k(T_{H_x})={\pi_x}_*\Big(\Big[\ch\big(\ev_x^*T_X/T_{\pi_x}(-\sigma_x)\big)\cdot\td(T_{\pi_x})\Big]_{k+1}\Big).
$$

Denote by $W_k$ the codimension $k$ part of the cycle 
\begin{gather*}
\ch\big(\ev_x^*T_X/T_{\pi_x}(-\sigma_x)\big)\cdot\td(T_{\pi_x})= \\
\big(\ev_x^*\ch(T_X)-\ch(T_{\pi_x})\big)\cdot\ch\big(\o_{U_x}(-\sigma_x)\big)\cdot\td(T_{\pi_x}).
\end{gather*}

Denote by $Z_k$ the codimension $k$ part of the cycle 
\begin{equation*}
\big(\ev_x^*\ch(T_X)-\ch(T_{\pi_x})\big)\cdot\ch\big(\o_{U_x}(-\sigma_x)\big),
\end{equation*}
Then $\ch_k(H_x)={\pi_x}_*\big(W_{k+1}\big)$, and $W_{k+1}=\sum_{j=0}^{k+1}Z_{k+1-j}\cdot\big[\td(T_{\pi_x})\big]_j$. 
We have:

$$\ev_x^*\big(\ch(T_X)\big)=\ev_x^*\big(n+c_1(X)+\ch_2(X)+\ch_3(X)+\cdots\big),$$

$$\ch\big(\o_{U_x}(-\sigma_x)\big)=\sum_{k=0}^{\infty}\frac{(-1)^k}{k!}\sigma_x^k,$$

$$\ch(T_{\pi_x})=\sum_{k=0}^{\infty}\frac{1}{k!}c_1(T_{\pi_x})^k,\quad 
\td(T_{\pi_x})=\sum_{k=0}^{\infty}\frac{(-1)^kB_k}{k!}c_1(T_{\pi_x})^k.$$

From \eqref{euler} and \eqref{L}, $c_1(T_{\pi_x})=2\sigma_x+\pi_x^*c_1(L_x)$. 
By repeatedly using \ref{describing_Lx}(iii), we have the following identities:

\begin{itemize}
\item[(iv) ] $\pi_x^*c_1(L_x)^i\cdot\sigma_x^j=(-1)^i\sigma_x^{i+j}$, for any $j\geq1$,
\item[(v) ]  $c_1(T_{\pi_x})^i\cdot\sigma_x^j=\sigma_x^{i+j}$, for any $j\geq1$, 
\item[(vi) ] $c_1(T_{\pi_x})^i=\left\{ 
\begin{aligned}
&\pi_x^*c_1(L_x)^i, &\text { if } i \text{ is even}\\
&2\sigma_x^i+\pi_x^*c_1(L_x)^i, &\text { if } i \text{ is odd.}
\end{aligned}
\right.$
\end{itemize} 

\begin{claim}\label{Z formula}
 
For any $k\geq1$ we have the following formulas:

\begin{equation}\label{Z}
Z_k=\ev_x^*\ch_k(X)+\frac{(n+1)(-1)^k}{k!}\sigma^k-\frac{1}{k!}\pi_x^*c_1(L_x)^k,
\end{equation}
\begin{equation}\label{Z times sigma}
Z_k\cdot\sigma_x=\frac{(n+1)(-1)^k}{k!}\sigma_x^{k+1}-\frac{1}{k!}\sigma_x\cdot\pi_x^*c_1(L_x)^k,
\end{equation}
\begin{equation}\label{push forward Z}
{\pi_x}_*Z_k={\pi_x}_*\ev_x^*\ch_k(X)-\frac{(n+1)}{k!}c_1(L_x)^{k-1},
\end{equation}
\begin{equation}\label{push forward Z times sigma}
{\pi_x}_*\big(Z_k\cdot\sigma_x\big)=\frac{n}{k!}c_1(L_x)^k.
\end{equation}
\end{claim}

In addition, $Z_0=n-1$ (formula (\ref{Z}) does not hold for $k=0$). 

\begin{proof}[Proof of Claim~\ref{Z formula}]
For $k\geq1$ we have:
$$Z_k=\sum_{j=0}^k\big(\ev_x^*\ch_j(X)-\frac{1}{j!}\pi_x^*c_1(T_{\pi_x})^j\big)\cdot\frac{(-1)^{k-j}}{(k-j)!}\sigma_x^{k-j}.$$
By identity \ref{describing_Lx}(ii), 
$$Z_k=\ev_x^*\ch_k(X)-\sum_{j=0}^k\frac{(-1)^{k-j}}{j!(k-j)!}\pi_x^*c_1(T_{\pi_x})^j\cdot\sigma_x^{k-j}.$$
Formula (\ref{Z}) follows now from identities (v), (vi) and $\sum_{j=0}^k\frac{(-1)^{k-j}}{j!(k-j)!}=0$. Formula (\ref{Z times sigma}) follows immediately from
(\ref{Z}) and \ref{describing_Lx}(ii).

Using the identity (iv) and the projection formula, we have
\begin{itemize}
\item[(vii) ] ${\pi_x}_*\sigma_x^{k}=(-1)^{k-1}c_1(L_x)^{k-1}$, for any $k\geq1$, and
\item[(viii) ] ${\pi_x}_*\pi_x^*(\gamma)=0$ for any class $\gamma\in A(H_x)$.
\end{itemize}
Formulas (\ref{push forward Z}) and (\ref{push forward Z times sigma}) now follow from (vii) and (viii).
\end{proof}

For simplicity, we denote by $A_j$ the coefficient of $c_1(M)^j$ in the formula for the Todd class $\td(M)$ of a line bundle $M$, i.e.,
$A_j=\frac{(-1)^j}{j!}B_j$. Recall that $A_0=1$, $A_1=1/2$, $A_2=1/12$, $A_3=0$, $A_4=-1/720$, etc.  

We have
$W_{k+1}=\sum_{j=0}^{k+1}A_{k+1-j}Z_j\cdot c_1(T_{\pi_x})^{k+1-j}$.
Since $A_l=0$ for all odd $l\geq3$, by identity (vi) the formula for $W_{k+1}$ becomes:
$$W_{k+1}=\sum_{j=0}^{k+1}A_{k+1-j}Z_j\cdot \pi_x^*c_1(L_x)^{k+1-j}+Z_k\cdot\sigma_x.$$

By the projection formula,  
$${\pi_x}_*W_{k+1}=\sum_{j=1}^{k+1}A_{k+1-j}\big({\pi_x}_*Z_j\big)\cdot c_1(L_x)^{k+1-j}+{\pi_x}_*\big(Z_k\cdot\sigma_x\big).$$

Using (\ref{push forward Z}), (\ref{push forward Z times sigma}), we have
\begin{gather*}
{\pi_x}_*W_{k+1}=\sum_{j=1}^{k+1}A_{k+1-j}{\pi_x}_*\ev_x^*\ch_j(X)\cdot c_1(L_x)^{k+1-j}+\\
-(n+1)\big(\sum_{j=1}^{k+1}\frac{A_{k+1-j}}{j!}\big)c_1(L_x)^k+\frac{n}{k!}c_1(L_x)^k.
\end{gather*}

It is easy to see that the identity $\sum_{l=0}^m B_l {{m+1}\choose{l}}=0$ implies the identity
\begin{equation*}
\sum_{j=1}^{k+1}\frac{A_{k+1-j}}{j!}=\frac{1}{k!}
\end{equation*}
(use $A_l=0$ for all odd $l\geq3$) and now \eqref{ch_k for H_x} follows.

To prove \eqref{c_1 of H_x} and \eqref{ch_2 of H_x} from \eqref{ch_k for H_x}, 
observe that ${\pi_x}_*\ev_x^*c_1(X)=d+2$.
\end{proof}


\section{Higher Fano manifolds} \label{section:proofs}

We start this section by recalling some results about the index and pseudoindex of a Fano manifold 
and extremal rays of its Mori cone. 
We refer to \cite{kollar_mori} for basic definitions and results about the minimal model program.

\begin{defn}\label{index}
Let $X$ be a Fano manifold. The \emph{index} of $X$ is the largest integer 
$r_X$ that divides $-K_X$ in $\pic(X)$.
The \emph{pseudoindex} of $X$ is the integer
$i_{X}=\min\big\{-K_Y\cdot C\ \big| \ C\subset Y \text{ rational curve }\big\}$.
\end{defn}

Notice that $1\leq r_X \leq i_X$.
Moreover 
$i_X\leq \dim X+1$, and $i_X=\dim X+1$ if and only if  $X\cong \p^n$
(see \ref{Hx}). 
By \cite{wisniewski_90}, if $\rho(X)>1$, then $i_X\leq \frac{\dim X}{2}+1$.

\begin{defn}\label{index&rays}
Let $X$ be a smooth complex projective variety.
Let $R$ be an extremal ray of the Mori cone $\nec{X}$, and let $f:Y\to Z$ be the corresponding contraction.
The \emph{exceptional locus} $E(R)$ of $R$ is the closed subset of $X$ where $f$ fails to be a local isomorphism.
Given a divisor $L$ on $X$, we set 
$L\cdot R= \min\big\{L\cdot C\ \big| \ C\subset X \text{ rational curve such that } [C]\in R\big\}$.
In particular, the \emph{length} of $R$ is 
$l(R)=-K_X\cdot R=\min\big\{-K_X\cdot C\ \big| \ C\subset X \text{ rational curve such that } [C]\in R\big\}\geq i_X$.
\end{defn}

Let $X$ be a Fano manifold, and 
$R$ an extremal ray of $\nec{X}$.
By the theorem on lengths of extremal rays, $l(R)\leq \dim X+1$.
By  \cite{AO_long_R}, if $\rho(X)>1$, then, 
\begin{equation} \label{long_R}
i_X\ +\ l(R)\ \leq \ \dim E(R) \ + \ 2.
\end{equation}

\begin{lemma}\label{adjunction}
Let $Y$ be a $d$-dimensional Fano manifold. Let $L$ be an ample divisor on $Y$ such 
that $-2K_Y-dL$ is ample. 
\begin{enumerate}
	\item Suppose that $(Y,L) \not\cong$ $\big(\p^d,\o(2)\big)$,  $\big(\p^1,\o(3)\big)$, 
		and let $R$ be any extremal ray of $\nec{Y}$. Then $L\cdot R=1$.
	\item Suppose that $\rho(Y)>1$. Then $(Y,L)$ is isomorphic to one of the following:
		\begin{enumerate} 
		\item $\Big(\p^m\times \p^m, p_{_1}^*\o(1)\otimes p_{_2}^*\o(1)\Big)$, with $d=2m$,
		\item $\Big(\p^{m+1}\times\p^{m} \ ,\  p_{_1}^*\o(1)\otimes p_{_2}^*\o(1)\Big)$, 
		with $d=2m+1$,
		\item $\Big(\p_{\p^{m+1}}\Big(\o(2)\oplus \o(1)^{^{\oplus m}}\Big) \ , \ \o_{\p}(1)\Big)$, 
		with $d=2m+1$,
		\item  $\Big(\p^{m}\times Q^{m+1} \ ,\  p_{_1}^*\o(1)\otimes p_{_2}^*\o(1)\Big)$, 
		with $d=2m+1$, or
		\item  $\Big(\p_{\p^{m+1}}\big(T_{\p^{m+1}}\big) \ , \ \o_{\p}(1)\Big)$, with $d=2m+1$.
		\end{enumerate}
\end{enumerate}
\end{lemma}

\begin{rem}
In (c) above, $Y$ can also be described as the blowup of $\p^{2m+1}$ along a linear subspace of dimension $m-1$.
In (e) above, $Y$ can also be described as a smooth divisor of type $(1,1)$ on $\p^{m+1}\times\p^{m+1}$.
\end{rem}

\begin{proof}[{Proof of Lemma~\ref{adjunction}}]
Suppose that $\rho(Y)=1$. 
Then $\nec{Y}$ consists of a single extremal ray  $R$, and there is an ample divisor 
$L'$ on $Y$ such that $\pic(X)=\z\cdot [L']$. 
Let $\lambda$ be the positive integer such that $L\sim \lambda L'$.
If $i_Y=d+1$, then $(Y,L')\cong \big(\p^d, \o_{\p^d}(1)\big)$, and 
$-2K_Y-dL\sim \big(d(2-\lambda) +2\big)L'$.
Since this is ample, either $\lambda\leq 2$ or $(d,\lambda)=(1,3)$.
If $i_Y\leq d$, then
$
1\leq (-2K_Y-dL)\cdot R = 2i_Y-d\lambda (L'\cdot R)\leq d\big(2-\lambda (L'\cdot R)\big).
$
Hence, $\lambda=L'\cdot R=1$.

From now on we assume that $\rho(Y)>1$.
Then $d>1$ and $i_Y\geq \frac{d+1}{2}$.
Moreover, by \cite{wisniewski_90},  $r_Y\leq i_Y\leq \frac{d}{2}+1$. 
Let $R$ be any extremal ray of $\nec{Y}$. We claim that $L\cdot R=1$.
Indeed, if $L\cdot R\geq 2$, then $l(R)=d+1$, contradicting \eqref{long_R}.

Suppose that $d=2m$ is even. Then $i_Y= m+1$.
Set $A=-K_Y-mL$. By assumption $A$ is ample.
For any extremal ray $R\subset \nec{Y}$, \eqref{long_R} implies that $l(R)=m+1$, 
and thus $A\cdot R= 1=L\cdot R$. 
Hence, $A\equiv L$, and so $A\sim L$ since $Y$ is Fano.
In particular $-K_Y\sim (m+1)L$, and thus $r_Y=m+1$.
By  \cite[Theorem B]{wisniewski_90}, this implies that $Y\cong \p^{m}\times \p^{m}$.

Now suppose that $d=2m+1$ is odd. Then $i_Y= m+1$.
Set $A'=-2K_Y-(2m+1)L$. By assumption $A'$ is ample.
Let $R$ be an extremal ray of $\nec{Y}$. 
Then $l(R)\geq m+1$, and $\dim E(R)\leq 2m+1$. 
By \eqref{long_R}, there are three possibilities:
\begin{enumerate}
	\item[(a)] $l(R)=  m+2$, $E(R)=Y$, and equality holds in \eqref{long_R};
	\item[(b)] $l(R)= m+1$, $\dim E(R) =2m$, and equality holds in \eqref{long_R}; or
	\item[(c)] $l(R)= m+1$, $E(R)=Y$, and equality in \eqref{long_R} fails by $1$.
\end{enumerate}

In \cite{AO_long_R}, Andreatta and Occhetta classify the cases in which equality holds in \eqref{long_R}, assuming
$\dim Y-1\leq\dim E(R)\leq \dim Y$. They show that in this case either $Y$ is a product of projective spaces, or a blowup
of $\p^{2m+1}$ along a linear subspace of dimension at most $m-1$. 
From this we see that in case (a) we must have $Y\cong \p^{m+1}\times\p^{m}$, 
while in case (b) $Y$ must be isomorphic to  
the blowup of $\p^{2m+1}$ along a linear subspace of dimension $m-1$.

From now on we assume that every extremal ray $R$ of $\nec{Y}$ falls into case (c) above,
which implies that $A'\cdot R= 1=L\cdot R$. Thus $A'\sim L$, 
$-K_Y\sim (m+1)L$, and thus $r_Y=m+1$.
By \cite{wisniewski_91}, this implies that either $Y\cong \p^{m}\times Q^{m+1}$, or
$Y\cong \p_{\p^{m+1}}\big(T_{\p^{m+1}}\big)$.
\end{proof}

\medskip

\begin{proof}[{Proof of Theorem~\ref{thm1}}]
Let $X$ be a Fano manifold with $\ch_2(X)\geq 0$.
Let $(H_x, L_x)$ be a polarized minimal family of rational curves through a general point $x\in X$.
 Set $d=\dim H_x$.
By Proposition~\ref{chern_characters},
$$
c_1(H_x)={\pi_x}_*\ev_x^*\big(\ch_2(X)\big)+\frac{d}{2}c_1(L_x).
$$ 
By Lemma~\ref{Tk_preserves_positivity}, 
${\pi_x}_*\ev_x^*$ preserves the properties of being weakly positive and nef.
Thus $-K_{H_x}$ is ample and $-2K_{H_x}-dL_x$ is nef.
Since $H_x$ is Fano, ${\pi_x}_*\ev_x^*\big(\ch_2(X)\big)$ is ample if and only if it is weakly positive.
Hence, $\ch_2(X)> 0$ implies that $-2K_{H_x}-dL_x$ is ample.
If ${\ev_x}_*\pi_x^*\big(\eff_1(H_x)\big)=\eff_2(X)$, and $-2K_{H_x}-dL_x$ is ample (respectively nef), then
clearly $\ch_2(X)$ is positive (respectively nef).
This proves the first part of the theorem.

The second part follows from Lemma~\ref{adjunction}(2). 
Examples of Fano manifolds $X$ with  $\ch_2(X)>0$ realizing
each of the exceptional pairs are given in section~\ref{examples}.

Finally, suppose that $\ch_2(X)>0$, $\ch_3(X)\geq 0$ and $d\geq 2$.
We already know from part (1) that  $-2K_{H_x}-dL_x$ is ample.
We want to prove that $\ch_2(H_x)>0$ and $\rho(H_x)=1$. 
For that purpose we may assume that $(H_x, L_x)\not\cong \big(\p^d,\o(2)\big)$.
Let $R\subset \eff_1(H_x)$ be an extremal ray.
By Lemma~\ref{adjunction}(1), there is a rational curve $\ell\subset H_x$ such that $R=\r_{\geq 0}[\ell]$
and  $L_x\cdot \ell=1$. Moreover,
$$
{\pi_x}_*\ev_x^*\big(2 \ \ch_2(X)\big)\cdot [\ell]=2\ \ch_2(X)\cdot {\ev_x}_*\pi_x^*[\ell]
= \big(c_1(T_X)^2-2c_2(T_X)\big)\cdot {\ev_x}_*\pi_x^*[\ell]
$$
is a positive integer, and thus $\geq 1$.
Therefore $\eta:={\pi_x}_*\ev_x^*\big(\ch_2(X)\big)-\frac{1}{2}c_1(L_x)\in N^1(H_x)_{\q}$ is nef.
We rewrite formula \eqref{ch_2 of H_x} of  Proposition~\ref{chern_characters} as
$$
\ch_2(H_x)={\pi_x}_*\ev_x^*\big(\ch_3(X)\big)+\frac{1}{2}\eta\cdot c_1(L_x)+\frac{d-1}{12}c_1(L_x)^2.
$$
Since ${\pi_x}_*\ev_x^*$ preserves the properties of being nef, 
we conclude that $\ch_2(H_x)$ is positive.
By \ref{non-examples},  none of the exceptional pairs $(H_x, L_x)$ from part (2)
satisfy $\ch_2(H_x)>0$.
Hence, $\rho(H_x)=1$.
\end{proof}

\begin{lemma} \label{Hx_covered_by_lines}
Let $X$ be a  Fano manifold. 
Let $(H_x, L_x)$ be a polarized minimal family of rational curves through a general point $x\in X$.
Set $d=\dim H_x$.
\begin{enumerate}
\item Suppose that $\ch_2(X)>0$, $d\geq 1$ and $(H_x, L_x)\not\cong$
$\big(\p^d,\o(2)\big)$,  $\big(\p^1,\o(3)\big)$.
Then any minimal dominating family of rational curves on $H_x$ 
parametrizes smooth rational curves of $L_x$-degree equal to $1$.
\item  Suppose that $\ch_2(X)>0$, $\ch_3(X)\geq0$, $d\geq 2$ and $(H_x, L_x)\not\cong$
$\big(\p^d,\o(2)\big)$.
Let $(W_h,M_h)$ be a polarized minimal family of rational curves through a general point $h\in H_x$.
Suppose that 
$(W_h,M_h)\not\cong$ $\big(\p^k,\o(2)\big)$,  $\big(\p^1,\o(3)\big)$.
Then there is an isomorphism $g:\p^2\to S\subset H_x$ mapping a point $p\in \p^2$ to $h\in H_x$,
sending lines through $p$ to curves parametrized by $W_h$, and such that 
$g^*L_x\cong \o_{\p^2}(1)$.
%
\end{enumerate} 
\end{lemma}

\begin{proof}
Suppose that $\ch_2(X)>0$, $d\geq 1$ and $(H_x, L_x)\not\cong$
$\big(\p^d,\o(2)\big)$,  $\big(\p^1,\o(3)\big)$.
Let $W$ be a minimal dominating family of rational curves on $H_x$.
We will show that the curves parametrized by $W$ have $L_x$-degree equal to $1$.
If $\rho(H_x)>1$, then this can be checked directly from the list in Theorem~\ref{thm1}(2).
So we assume $\rho(H_x)=1$.
Let $\ell\subset H_x$ be a curve parametrized by $W$. 
Then, as in \ref{Hx},
$-K_{H_x}\cdot \ell\leq d+1$, and $-K_{H_x}\cdot \ell= d+1$ if and only if $H_x\cong \p^d$.
By Theorem~\ref{thm1}(1), $-K_{H_x}\cdot \ell>\frac{d}{2} L_x\cdot \ell$.
If $ L_x\cdot \ell >1$, then $(H_x, L_x)\cong$
$\big(\p^d,\o(2)\big)$ or  $\big(\p^1,\o(3)\big)$, 
contradicting our assumptions. We conclude that $L_x\cdot \ell =1$.
The generically injective morphism 
$\tau_x:  H_x  \to \p(T_xX^{^{\vee}})$ defined in \ref{describing_Lx}
maps curves parametrized by $W$ to lines.
So all curves parametrized by $W$ are smooth.

Now suppose we are under the assumptions of the second part of the lemma. 
By Theorem~\ref{thm1}(3), $H_x$ is a Fano manifold with $\ch_2(H_x)>0$ and $\rho(H_x)=1$.
If $d=2$, then $H_x\cong \p^2$, and by our assumptions $L_x\cong \o(1)$.
Now suppose $d=3$.
Recall that the only Fano threefolds satisfying $\ch_2>0$ are $\p^3$ and the smooth quadric hypersurface $Q^3\subset \p^4$ (\cite{2Fano_3folds}).
The polarized minimal family of rational curves through a general point of $Q^3$ is isomorphic to $\big(\p^1,\o(2)\big)$.
So our assumptions imply that $(H_x, L_x)\cong \big(\p^3,\o(1)\big)$, and the conclusion of the lemma is clear.
Finally assume that $d\geq 4$.
By the first part of the lemma, 
$L_x\cdot \ell =1$ 
for any curve $\ell$ parametrized by $W_h$, and $W_h^{\text{Sing},h}=\emptyset$.
Theorem~\ref{thm1}(1) implies that $i_{H_x}>\frac{d}{2}\geq 2$, 
and thus $\dim W_h\geq i_{H_x}-2\geq 1$.
By the first part of the lemma, now applied to the variety $H_x$, $W_h$ is covered by smooth 
rational curves of $M_h$-degree equal to $1$.
By Lemma~\ref{f:P^k->X}, applied to the variety $H_x$, 
there is a generically injective morphism $g:(\p^2,p)\to (H_x, h)$ mapping 
lines  through $p$ to curves on $H_x$ parametrized by $W_h$.
Since these curves have $L_x$-degree equal to $1$, they
are mapped to lines by the generically injective morphism
$\tau_x:H_x \to  \p(T_xX^{^{\vee}})$.
We conclude that the composition $\tau_x\circ g:\p^2 \to \p(T_xX^{^{\vee}})$ is an isomorphism onto its image
and $(\tau_x\circ g)^*\o_{\p}(1) \cong \o_{\p^2}(1)$.
This proves the second part of the lemma.
\end{proof}

\begin{proof}[{Proof of Theorem~\ref{thm3}}]
Let the notation  and assumptions be as in Theorem~\ref{thm3}. 

Part (1) was proved in \cite{dJ-S:2fanos_2}. It also follows from Theorem~\ref{thm1}(1):
the conditions $\ch_2(X)\geq 0$ and $d\geq 1$ imply that $H_x$ is a positive dimensional 
Fano manifold, and thus covered by rational curves. For any rational curve $\ell\subset H_x$,
$S=\ev_x\big(\pi_x^{-1}(\ell)\big)$ is a rational surface on $X$ through $x$.
Notice that $S$ is covered by rational curves parametrized by $H_x$.

For part (2), assume $\ch_2(X)> 0$ and $(H_x, L_x)\not\cong$
$\big(\p^d,\o(2)\big)$,  $\big(\p^1,\o(3)\big)$.
If $d=1$, then $(H_x, L_x)\cong \big(\p^1,\o(1)\big)$. In this case the morphism
$\tau_x:  H_x  \to \p(T_xX^{^{\vee}})$ 
is an isomorphism onto its image, and thus $H_x^{\text{Sing},x}=\emptyset$ by 
\cite[Corollary 2.8]{artigo_tese}.
If $d>1$, let $W$ be a minimal dominating family of rational curves on $H_x$.
By Lemma~\ref{Hx_covered_by_lines}(1),
the curves parametrized by $W$ are smooth and have $L_x$-degree equal to $1$.
Moreover, since $H_x^{\text{Sing},x}$ is at most finite,
a general curve parametrized by $W$ 
is contained in $H_x\setminus  H_x^{\text{Sing},x}$ by \cite[II.3.7]{kollar}. 
Part (2) now follows from Lemma~\ref{f:P^k->X}.

From now on assume that $\ch_2(X)> 0$, $\ch_3(X)\geq 0$ and $d\geq 2$.
Then $H_x$ is a Fano manifold with $\rho(H_x)=1$ and $\ch_2(H_x)>0$
by Theorem~\ref{thm1}(3).
If $d=2$, then $H_x\cong\p^2$ and $U_x=\p(E_x)$ is a rational $3$-fold.
Hence, $\ev_x(U_x)$  is a rational $3$-fold through $x$ except possibly if $\ev_x$ 
fails to be birational onto its image.
This can only occur if $\cC_x$ is singular by \cite[Corollary 2.8]{artigo_tese},
in which case $L_x\cong \o(2)$.
Now suppose $d\geq 3$.
We claim that there is a rational surface $S\subset H_x\setminus H_x^{\text{Sing},x}$.
From this it follows that $\ev_x\big|_{\pi_x^{-1}(S)}$ is generically injective and $\ev_x\big(\pi_x^{-1}(S)\big)$ is a rational 
$3$-fold through $x$.
If $d=3$, then $H_x\cong \p^3$ or $Q^3\subset \p^4$,
and we can find a rational surface $S\subset H_x\setminus H_x^{\text{Sing},x}$.
Now assume $d\geq 4$, and let $W_h$ be a minimal family of rational curves through a general point $h\in H_x$.
Then $W_h$ is a Fano manifold and $\dim W_h\geq i_{H_x}-2\geq 1$.
As in the proof of part (1), now applied to $H_x$,  
each rational curve on $W_h$ yields a rational surface $S$ on $H_x$ through $h$.
Recall that $H_x^{\text{Sing},x}$ is at most finite. 
If $S \cap  H_x^{\text{Sing},x} \neq \emptyset$, then there is a rational curve on $H_x$ parametrized by $W_h$ 
meeting $H_x^{\text{Sing},x}$. 
If this holds for a general point $h\in H_x$, then there is a point $h_0\in H_x^{\text{Sing},x}$
that can be connected 
to a general point of $H_x$ by a curve parametrized 
by a suitable minimal dominating family of rational curves on $H_x$.
But this implies that a curve from this minimal dominating family has 
$-K_{H_x}$-degree equal to $d+1$, and so we must have $H_x\cong \p^d$ (see \ref{Hx}).
Since $d>2$, we can find a rational surface $S'\subset H_x\setminus H_x^{\text{Sing},x}$.
This proves part (3).

Finally, suppose we are under the assumptions of part (4).
By Lemma~\ref{Hx_covered_by_lines}(2), $H_x$ is covered by surfaces 
$S$ such that $(S, L_x|_S)\cong (\p^2,\o_{\p^2}(1))$. 
Exactly as in the proof of part (3) above, 
we can take such a surface $S\subset H_x\setminus H_x^{\text{Sing},x}$.
Part (4) now follows from Lemma~\ref{f:P^k->X}.
\end{proof}


\section{Examples}\label{examples}

In this section we discuss examples of Fano manifolds $X$ with $\ch_2(X)\geq 0$.
Theorem~\ref{thm1} provides a new way of checking positivity of $\ch_2(X)$,
enabling us to find new examples. 
Examples \ref{C.I.} and \ref{G} below appear in \cite{dJ-S:2fanos_1}. 
Example \ref{H_in_G} does not appear explicitly in \cite{dJ-S:2fanos_1}, but it can be inferred
from \cite[Theorem 1.1(3)]{dJ-S:2fanos_1}.
Examples \ref{OG}, \ref{SG}, \ref{degenerate SG} and \ref{G2/P} are new.

\begin{say}[Complete Intersections]\label{C.I.}
Let $X$  be a complete intersection of type $(d_1,\ldots, d_c)$ in $\p^n$. 
Standard Chern class computations show that $\ch_k(X)>0$ (respectively $\geq0$)
if and only if  $\sum d_i^k\leq n$ (respectively $\leq n+1$). See for instance 
\cite[2.1 and 2.4]{dJ-S:2fanos_1}.

Let $x\in X$ be a general point, and let $H_x$ be the variety of 
lines  through $x$ on $X$.
Then $H_x$ is a complete intersection of type 
$(1,2,\ldots d_1,\ldots,1,2,\ldots d_c)$  in $\p^{n-1}$, and $L_x\cong \o(1)$. 
The condition from Theorem~\ref{thm1}(1)  of $-2K_{H_x}-dL_x$ being ample 
(respectively nef) is clearly equivalent to 
$\sum d_i^2\leq n$ (respectively $\leq n+1$). 
\end{say}

\begin{say}[Grassmannians]  \label{G}
Let $X=G(k,n)$ be the Grassmannian of $k$-dimensional linear subspaces of an $n$-dimensional vector 
space $V$, with $2\leq k\leq\frac{n}{2}$. 
As computed in \cite[2.2]{dJ-S:2fanos_1}, the second Chern class of $X$ is given by 
$$
\ch_2(X)=\frac{n+2-2k}{2}\sigma_2-\frac{n-2-2k}{2}\sigma_{1,1},
$$
where $\sigma_2$ and $\sigma_{1,1}$ are the usual Schubert cycles of codimension $2$. 
Recall that $\eff_{2}(X)$ is generated by the dual Schubert cycles $\sigma_{1,1}^{*}$ and $\sigma_2^{*}$.
Thus $\ch_2(X)>0$ (respectively $\geq0$) if and only if $2k\leq n \leq2k+1$ (respectively $2k\leq n \leq2k+2$).

Given $x\in X$, let $H_x$ be the variety of  lines  through $x$ on $X$ under the Pl\"ucker embedding. 
As explained in \cite[1.4.4]{hwang}, 
$(H_x,L_x)\cong \big(\p^{k-1}\times\p^{n-k-1}, p_{_1}^*\o(1)\otimes p_{_2}^*\o(1)\big)$.
Indeed, if $x$ parametrizes a linear subspace $[W]$, then 
a line through $x$ corresponds to subspaces $U$ and $U'$ of $V$, of dimension $k-1$ and $k+1$, 
such that $U\subset W\subset U'$. So there is a natural identification $H_x\cong\p(W)\times\p(V/W)^*$.

The condition of $-2K_{H_x}-dL_x$ being ample (respectively nef) is clearly equivalent to $2k\leq n\leq 2k+1$ (respectively $2k\leq n\leq 2k+2$). 
Notice also that the map $T_1: \eff_1(H_x) \to\eff_{2}(X)$ sends
lines on fibers of $p_1$ and $p_2$  to the dual Schubert cycles $\sigma_2^{*}$ and $\sigma_{1,1}^{*}$.
In particular it is surjective.

Exceptional pairs (a), (b) in Theorem~\ref{thm1}(2) occur in this case.
\end{say}

\begin{say}[Hyperplane sections of Grassmannians] \label{H_in_G}
Let $X$ be a general hyperplane section of the Grassmannian $G(k,n)$ under the Pl\"ucker embedding,
where $2\leq k\leq\frac{n}{2}$. 
Let $x\in X$ be a general point, and $H_x$ the variety of lines through $x$ on $X$. 
Then $H_x$ is a smooth divisor of type $(1,1)$ 
in $\p^{k-1}\times\p^{n-k-1}$ and $L_x$ is the restriction to $H_x$ of $p_{_1}^*\o(1)\otimes p_{_2}^*\o(1)$. 
Thus  $-2K_{H_x}-dL_x$ is ample (respectively nef) if and only if $n=2k$
(respectively $2k\leq n\leq 2k+1$). 
In these cases $T_1: \eff_1(H_x) \to\eff_{2}(X)$ is surjective,
and thus Theorem~\ref{thm1}(1) applies.  We conclude that 
$\ch_2(X)>0$ (respectively $\geq0$) if and only if $n=2k$ (respectively $2k\leq n\leq 2k+1$). 

This example occurs as the exceptional case (e) in Theorem~\ref{thm1}(2).
\end{say}

\begin{say}[Orthogonal Grassmannians]\label{OG}
We fix $Q$ a nondegenerate symmetric bilinear form on the $n$-dimensional vector 
space $V$, and $k$ an integer satisfying $2\leq k<\frac{n}{2}-1$.
Let $X=OG(k,n)$ be the subvariety of the 
Grassmannian $G(k,n)$ parametrizing linear subspaces that are isotropic with respect to $Q$.
Then $X$ is a Fano manifold of dimension $\frac{k(2n-3k-1)}{2}$ and $\rho(X)=1$.
Notice that $X$ is the zero locus in $G(k,n)$ of a global section of the vector bundle
$\sym^2(\cS^*)$, where $\cS^*$ is the universal quotient bundle on  $G(k,n)$.
Using this description and the formula for $\ch_2\big(G(k,n)\big)$ described in \ref{G}, 
standard Chern class computations show that
$$
\ch_2(X)=\frac{n-1-3k}{2}\sigma_2-\frac{n-3-3k}{2}\sigma_{1,1},
$$
where we denote by the same symbols  $\sigma_2$ and $\sigma_{1,1}$ the restriction to $X$ 
of the corresponding Schubert cycles on  $G(k,n)$.

Given $x\in X$, let $H_x$ be the variety of  lines  through $x$ on $X$ under the Pl\"ucker embedding.
We claim that $(H_x,L_x)\cong \big(\p^{k-1}\times Q^{n-2k-2}, p_{_1}^*\o(1)\otimes p_{_2}^*\o(1)\big)$.
Indeed, if $x$ parametrizes a linear subspace $[W]$, then 
a line through $x$ on $X$ corresponds to a pair $(U,U')\in\p(W)\times\p(V/W)^*$ such  that
$U'\subset U^{\perp}$ and $Q(v,v)=0$ for any $v\in U'$. 
This is equivalent to the condition that $U'\subset W^{\perp}$ and $Q(v,v)=0$ for any $v\in U'$. 
The form $Q$ induces a nondegenerate quadratic form on $W^{\perp}/W$, which 
defines a smooth quadric $Q^{n-2k-2}$ in $\p(W^{\perp}/W)^*\cong \p^{n-2k-1}$. 
The condition then becomes $U'\subset W^{\perp}$ and $[U'/W]\in Q^{n-2k-2}$, proving the claim.
Thus  $-2K_{H_x}-dL_x$ is ample (respectively nef) if and only if $n=3k+2$ 
(respectively $3k+1\leq n\leq 3k+3$). 
In these cases there are lines on fibers of  $p_1$ and $p_2$ contained in $H_x\subset \p^{k-1}\times\p^{n-k-1}$, and thus
the composite map $\eff_1(H_x) \to\eff_{2}(X)\into \eff_{2}\big(OG(k,n)\big)$ is surjective.
Thus $T_1: \eff_1(H_x) \to\eff_{2}(X)$ is surjective, and Theorem~\ref{thm1}(1) applies. 
We conclude that 
$\ch_2(X)>0$ (respectively $\geq0$) if and only if $n=3k+2$ 
(respectively $3k+1\leq n\leq 3k+3$). 

The exceptional pair (d) in Theorem~\ref{thm1}(2) occurs in this case.
\end{say}

\begin{say}[Symplectic Grassmannians] \label{SG}
We fix $\omega$ a non-degenerate antisymmetric bilinear form on the $n$-dimensional vector 
space $V$, $n$ even, and $k$ an integer satisfying $2\leq k\leq\frac{n}{2}$.
Let $X=SG(k,n)$ be the subvariety of the 
Grassmannian $G(k,n)$ parametrizing linear subspaces that are isotropic with respect to $\omega$. 
Then $X$ is a Fano manifold of dimension $\frac{k(2n-3k+1)}{2}$ and $\rho(X)=1$.
Notice that $X$ is the zero locus in $G(k,n)$ of a global section of the vector bundle
$\wedge^2(\cS^*)$, where $\cS^*$ is the universal quotient bundle on  $G(k,n)$.
Using this description and the formula for $\ch_2\big(G(k,n)\big)$ described in \ref{G}, 
standard Chern class computations show that
$$
\ch_2(X)=\frac{n+3-3k}{2}\sigma_2-\frac{n+1-3k}{2}\sigma_{1,1},
$$
where we denote by the same symbols  $\sigma_2$ and $\sigma_{1,1}$ the restriction to $X$ 
of the corresponding Schubert cycles on  $G(k,n)$.

Given $x\in X$, let $H_x\subset \p^{k-1}\times\p^{n-k-1}$ be the variety of  lines  through $x$ on $X$ under the Pl\"ucker embedding.
By  \cite[1.4.7]{hwang}, 
$(H_x,L_x)\cong \big( \p_{\p^{k-1}}(\o(2)\oplus\o(1)^{n-2k}), \o_{\p}(1)  \big)$.
When $n=2k$ this becomes $(H_x,L_x)\cong \big(\p^{k-1},\o(2)\big)$.
When $n>2k$, $H_x$ can also be described as 
the blow-up of $\p^{n-k-1}$ along a linear subspace $\p^{n-2k-1}$, and $L_x$ as $2H-E$, where 
$H$ is the hyperplane class in $\p^{n-k-1}$ and $E$ is the exceptional divisor.
Thus  $-2K_{H_x}-dL_x$ is ample (respectively nef) if and only if $n=2k$ or $n=3k-2$ 
(respectively $n=2k$ or $3k-3\leq n\leq 3k-1$). 
In these cases $T_1: \eff_1(H_x) \to\eff_{2}(X)$ is surjective.
Indeed, if $n=2k$, then $b_4(X)=1$. 
If $3k-3\leq n\leq 3k-1$, then there are lines on fibers of  $p_1$ and $p_2$ contained in $H_x\subset \p^{k-1}\times\p^{n-k-1}$, and thus
the composite map $\eff_1(H_x) \to\eff_{2}(X)\into \eff_{2}\big(OG(k,n)\big)$ is surjective.
So Theorem~\ref{thm1}(1) applies, and we conclude that 
$\ch_2(X)>0$ (respectively $\geq0$) if and only if $n=2k$ or $n=3k-2$ 
(respectively $n=2k$ or $3k-3\leq n\leq 3k-1$).

When $m$ is even, the exceptional pair (c) in Theorem~\ref{thm1}(2) occurs for 
$X=SG(m+2,3m+4)$. 
The exceptional pair $(H_x,L_x)\cong (\p^d, \o(2))$ in Theorem~\ref{thm3}(2)
occurs for  $X=SG(d+1,2d+2)$. 
\end{say}

\begin{say}[A two-orbit variety]\label{degenerate SG}
We fix $\omega$ an antisymmetric bilinear form of maximum rank $n-1$ on the $n$-dimensional vector 
space $V$, $n$ odd, and $k$ an integer satisfying $2\leq k<\frac{n}{2}$.
Let $X$ be the subvariety of the 
Grassmannian $G(k,n)$ parametrizing linear subspaces 
that are isotropic with respect to $\omega$. 
Then $X$ is a Fano manifold of  dimension $\frac{k(2n-3k+1)}{2}$ and $\rho(X)=1$.
Note that $X$ is not  homogeneous. 

The same argument presented in \ref{SG} above, taking $x\in X$ a general point, 
shows that $\ch_2(X)>0$ (respectively $\ch_2(X)\geq0$)
if and only if $n=3k-2$ (respectively  $3k-3\leq n\leq 3k-1$).
When $m$ is odd, the exceptional pair (c) in Theorem~\ref{thm1}(2) occurs for such $X$, with 
$k=m+2$ and $n=3m+4$. 
\end{say}

\begin{say}[The $5$-dimensional homogeneous space $G_2/P$]  \label{G2/P}
Let $X$ be the $5$-dimensional homogeneous space $G_2/P$.
Then $X$ is a Fano manifold with $\rho(X)=1$, and 
$(H_x,L_x)\cong\big(\p^1,\o(3)\big)$, as explained in \cite[1.4.6]{hwang}.
Since $b_4(X)=1$, the map $T_1: \eff_1(H_x) \to\eff_{2}(X)$ is surjective,
and thus Theorem~\ref{thm1}(1) applies.  We conclude that 
$\ch_2(X)>0$.
The exceptional pair $(H_x,L_x)\cong (\p^1, \o(3))$ in Theorem~\ref{thm3}(2)
occurs in this case.
\end{say}

\begin{say}[Non-Examples]\label{non-examples}
By \cite[Theorem 1.2]{dJ-S:2fanos_1}, the following smooth projective varieties do not satisfy $\ch_2(X)>0$.
\begin{itemize}
\item Products $X\times Y$, with $\dim X, \dim Y>0$.
\item Projective space bundles $\p(E)$, with $\dim X>0$ and $\rank E \geq2$.
\item Blowups of $\p^n$ along smooth centers of codimension $2$.
\end{itemize}

By Theorem \ref{thm2}(1),
if $X$ a Fano manifold and $H_x$ is not Fano, then $\ch_2(X)$ is not nef. 
This is the case, for instance, when $X$ is the moduli space of rank $2$ vector bundles with fixed determinant 
of odd degree on a smooth curve $C$ of  genus $\geq2$. 
In this case $H_x$ is the family of Hecke curves through $x=[E]\in X$, which are 
conics with respect to the ample generator of $\pic(X)$. 
As explained in \cite[1.4.8]{hwang}, $H_x\cong\p_C(E)$, which is not Fano.
\end{say}

\bibliographystyle{amsalpha}
\bibliography{carolina}

\end{document}